\documentclass[a4paper,12pt]{article}
\usepackage{times, url}
\usepackage{amsmath,amssymb,amsthm,amsfonts}

\theoremstyle{theorem}
\newtheorem{theorem}{Theorem}

\theoremstyle{definition}

\theoremstyle{lemma}

\theoremstyle{corollary}

\theoremstyle{example}

\begin{document}

\begin{center}
{\LARGE \bf Prime-Generating Polynomial}
\vspace{6mm}

{\large \bf Madieyna Diouf}\\

e-mail: \url{mdiouf1@asu.edu}
\vspace{2mm}



\end{center}
\vspace{4mm}

\noindent
{\bf Abstract:} We present a prime-generating polynomial $(1+2n)(p -2n) + 2$ where $p>2$ is a lower member of a pair of twin primes less than $41$ and the integer $n$ is such that $\: \frac {1-p}{2} < n < p-1$. 
\vspace{4mm}

\section{Introduction} 
We consecutively generate prime numbers from a low degree polynomial in which the coefficients are no more than two digits. The best-known polynomial that generates (possibly in absolute value) only primes is $n^2 +n+41$ due to Euler\cite{euler1772} which gives distinct primes for the $40$ consecutive integers $0$ to $39$. These numbers are called Euler numbers \cite{flannery}. Moreover, if  $f(n)$ is prime-generating for , $0\leq n \leq x$, then so is $f(x-n)$. Thus, the function $(n-40)^2 + (n-40) + 41$ generates primes for $80$ consecutive integers corresponding to the $40$ primes above where each is duplicated \cite{Hardy}. We observe that the polynomial $(1+2n)(p -2n) + 2 = -4n^2 + 2n(p-1) + p +2$ (possibly in absolute value) generates primes $85$ times.
\section{Initial Observations}
The decision to investigate the integers $(1+2n)(p -2n) + 2$ rises from an observation of the pair of twin primes $(p, p_ 0)$ where $p < 41.$ We obtain the relation, 
\begin {equation}
  p_ 0 = 1.p +2.
\end{equation}
We say that $p_0$ is obtained by adding $2$ to the product of the smallest positive integer (that is $1$) and the largest prime less than $p_0$ (that is $p).$ Consider the product $1.p$ from equation $(1)$, we ask what if we add $2$ to the factor $1$ and subtract $2$ from the factor $p$, take the product of the two obtained factors (after the addition) and add $2$ to the result? This leads to the following.
\begin {equation}
  p_ 1 = (1+2)(p-2) +2.
\end{equation}
Would $p_1$ be a prime? Say that we repeat the process above by adding $2$ and subtracting $2$ respectively from left and right side of the product in equation $(2)$. This gives rise to
\begin {equation}
  p_ 2 = (1+2+2)(p-2-2) +2.
\end{equation}
Would $p_2$ be a prime? If we continue the process, then we can consider the following formula.
\begin {equation}
  p_ n = (1+2n)(p-2n) +2.
\end{equation}
What are the consecutive values of $n$ for which $p_n$ is a prime?
\begin{theorem}
If $p>2$ is a lower member of a pair of twin primes less than $41$, then 
\begin{equation}
p_n =  (1+2n)(p -2n) + 2 \textup{\: \: } (possibly\:  in\:  absolute \: value)
\end{equation}
is a prime for every integer $n$ in the interval $(\frac {1-p}{2}, \: p-1).$
\end{theorem} 
\begin{proof}
Let $p$ be the lower member of the pair of twin primes (3, 5); thus, $p$ = 3. And 
\newline
 $ p_n = \ \mid (1+2n)(p -2n) + 2 \mid$ implies that $ p_n = \mid (1+2n)(3 -2n) + 2 \mid.$ 
\newline
And \ \ \ \ \ \ \ \ \ \ \ \  $ \frac {1-p}{2} < n < p-1$ implies that $ -1 < n < 2$.
\newline
For $n$ = 0 $ p_n =\  \mid (1+2n)(3 -2n) + 2 \mid$ implies that  $ p_n = \mid 5 \mid = 5.$ \newline
For $n$ = 1  $ p_n =\  \mid (1+2n)(3 -2n) + 2 \mid$ implies that  $ p_n = \mid 5 \mid = 5.$\newline
\newline
For the pair of twin primes (5, 7), we have $p$ = 5. And 
\newline
 $ p_n = \ \mid (1+2n)(p -2n) + 2 \mid$ implies that $ p_n = \mid (1+2n)(5 -2n) + 2 \mid$.
\newline
And \ \ \ \ \ \ \ \ \ \ \ \  $ \frac {1-p}{2} < n < p-1$  implies that $ -2 < n < 4$.
\newline
For $n$ = -1 $ p_n =\  \mid (1+2n)(5 -2n) + 2 \mid$ implies that  $ p_n = \mid -5 \mid = 5.$ \newline
For $n$ = 0 $ p_n =\  \mid (1+2n)(5 -2n) + 2 \mid$ implies that  $ p_n = \mid 7 \mid = 7.$ \newline
For $n$ = 1  $ p_n =\  \mid (1+2n)(5 -2n) + 2 \mid$ implies that  $ p_n = \mid 11 \mid = 11.$\newline 
For $n$ = 2 $ p_n =\  \mid (1+2n)(5 -2n) + 2 \mid$ implies that  $ p_n = \mid 7\mid = 7$. \newline
For $n$ = 3 $ p_n =\  \mid (1+2n)(5 -2n) + 2 \mid$ implies that  $ p_n = \mid -5 \mid = 5.$ \newline
\newline
For the pair of twin primes (11, 13), we have $p$ = 11. And 
\newline
 $ p_n = \ \mid (1+2n)(p -2n) + 2 \mid$ implies that $ p_n = \mid (1+2n)(11 -2n) + 2 \mid$.
\newline
And \ \ \ \ \ \ \ \ \ \ \ \  $ \frac {1-p}{2} < n < p-1$ implies that $ -5 < n < 10$.
\newline
For $n$ = -4 $ p_n =\  \mid (1+2n)(11 -2n) + 2 \mid$ implies that  $ p_n = \mid -131 \mid = 131.$ \newline
For $n$ = -3 $ p_n =\  \mid (1+2n)(11 -2n) + 2 \mid$ implies that  $ p_n = \mid -83 \mid = 83.$ \newline
For $n$ = -2  $ p_n =\  \mid (1+2n)(11 -2n) + 2 \mid$ implies that  $ p_n = \mid -43 \mid = 43.$\newline
For $n$ = -1 $ p_n =\  \mid (1+2n)(11 -2n) + 2 \mid$ implies that  $ p_n = \mid -11\mid = 11.$ \newline
For $n$ = 0 $ p_n =\  \mid (1+2n)(11 -2n) + 2 \mid$ implies that  $ p_n = \mid 13 \mid = 13.$ \newline
For $n$ = 1 $ p_n =\  \mid (1+2n)(11 -2n) + 2 \mid$ implies that  $ p_n = \mid 29 \mid = 29.$ \newline
For $n$ = 2  $ p_n =\  \mid (1+2n)(11 -2n) + 2 \mid$ implies that  $ p_n = \mid 37 \mid = 37.$\newline
For $n$ = 3 $ p_n =\  \mid (1+2n)(11 -2n) + 2 \mid$ implies that  $ p_n = \mid 37\mid = 37.$ \newline
For $n$ = 4 $ p_n =\  \mid (1+2n)(11 -2n) + 2 \mid$ implies that  $ p_n = \mid 29 \mid = 29.$ \newline
For $n$ = 5 $ p_n =\  \mid (1+2n)(11 -2n) + 2 \mid$ implies that  $ p_n = \mid 13 \mid = 13.$ \newline
For $n$ = 6 $ p_n =\  \mid (1+2n)(11 -2n) + 2 \mid$ implies that  $ p_n = \mid -11 \mid = 11.$ \newline
For $n$ = 7 $ p_n =\  \mid (1+2n)(11 -2n) + 2 \mid$ implies that  $ p_n = \mid -43 \mid = 43.$ \newline
For $n$ = 8  $ p_n =\  \mid (1+2n)(11 -2n) + 2 \mid$ implies that  $ p_n = \mid -83 \mid = 83.$\newline
For $n$ = 9 $ p_n =\  \mid (1+2n)(11 -2n) + 2 \mid$ implies that  $ p_n = \mid -131\mid = 131.$ \newline
\newline
For the pair of twin primes (17, 19), we have $p$ = 17. And 
\newline
 $ p_n = \ \mid (1+2n)(p -2n) + 2 \mid$ implies that $ p_n = \mid (1+2n)(17 -2n) + 2 \mid$. 
\newline
And \ \ \ \ \ \ \ \ \ \ \ \  $ \frac {1-p}{2} < n < p-1$ implies that $ -8 < n < 16$.
\newline
For $n$ = -7 $ p_n =\  \mid (1+2n)(17 -2n) + 2 \mid$ implies that  $ p_n = \mid -401 \mid = 401.$ \newline
For $n$ = -6 $ p_n =\  \mid (1+2n)(17 -2n) + 2 \mid$ implies that  $ p_n = \mid -317 \mid = 317$. \newline
For $n$ = -5  $ p_n =\  \mid (1+2n)(17 -2n) + 2 \mid$ implies that  $ p_n = \mid 241 \mid = 241.$\newline 
For $n$ = -4 $ p_n =\  \mid (1+2n)(17 -2n) + 2 \mid$ implies that  $ p_n = \mid -173\mid = 173$. \newline
For $n$ = -3 $ p_n =\  \mid (1+2n)(17 -2n) + 2 \mid$ implies that  $ p_n = \mid -113 \mid = 113$. \newline
For $n$ = -2 $ p_n =\  \mid (1+2n)(17 -2n) + 2 \mid$ implies that  $ p_n = \mid -61 \mid = 61$. \newline
For $n$ = -1  $ p_n =\  \mid (1+2n)(17 -2n) + 2 \mid$ implies that  $ p_n = \mid -17 \mid = 17$.\newline 
For $n$ = 0 $ p_n =\  \mid (1+2n)(17 -2n) + 2 \mid$ implies that  $ p_n = \mid 19\mid = 19.$ \newline
For $n$ = 1 $ p_n =\  \mid (1+2n)(17 -2n) + 2 \mid$ implies that  $ p_n = \mid 47 \mid = 47.$ \newline
For $n$ = 2 $ p_n =\  \mid (1+2n)(17 -2n) + 2 \mid$ implies that  $ p_n = \mid 67 \mid = 67$. \newline
For $n$ = 3 $ p_n =\  \mid (1+2n)(17 -2n) + 2 \mid$ implies that  $ p_n = \mid 79 \mid = 79$. \newline
For $n$ = 4 $ p_n =\  \mid (1+2n)(17 -2n) + 2 \mid$ implies that  $ p_n = \mid 83 \mid = 83.$ \newline
For $n$ = 5  $ p_n =\  \mid (1+2n)(17 -2n) + 2 \mid$ implies that  $ p_n = \mid 79\mid = 79$.\newline 
For $n$ = 6 $ p_n =\  \mid (1+2n)(17 -2n) + 2 \mid$ implies that  $ p_n = \mid 67\mid = 67$. \newline
For $n$ = 7 $ p_n =\  \mid (1+2n)(17 -2n) + 2 \mid$ implies that  $ p_n = \mid 47 \mid = 47$. \newline
For $n$ = 8 $ p_n =\  \mid (1+2n)(17 -2n) + 2 \mid$ implies that  $ p_n = \mid 19 \mid = 19.$ \newline
For $n$ = 9 $ p_n =\  \mid (1+2n)(17 -2n) + 2 \mid$ implies that  $ p_n = \mid -17 \mid = 17.$ \newline
For $n$ = 10 $ p_n =\  \mid (1+2n)(17 -2n) + 2 \mid$ implies that  $ p_n = \mid -61 \mid = 61.$ \newline
For $n$ = 11  $ p_n =\  \mid (1+2n)(17 -2n) + 2 \mid$ implies that  $ p_n = \mid -113 \mid = 113.$\newline 
For $n$ = 12 $ p_n =\  \mid (1+2n)(17 -2n) + 2 \mid$ implies that  $ p_n = \mid -173\mid = 173.$ \newline
For $n$ = 13 $ p_n =\  \mid (1+2n)(17 -2n) + 2 \mid$ implies that  $ p_n = \mid -241 \mid = 241$. \newline
For $n$ = 14 $ p_n =\  \mid (1+2n)(17 -2n) + 2 \mid$ implies that  $ p_n = \mid -317 \mid = 317.$ \newline
For $n$ = 15 $ p_n =\  \mid (1+2n)(17 -2n) + 2 \mid$ implies that  $ p_n = \mid -401 \mid = 401.$\newline 
\newline
For the pair of twin primes (29, 31), we have $p$ = 29. And 
\newline
 $ p_n = \ \mid (1+2n)(p -2n) + 2 \mid$ implies that $ p_n = \mid (1+2n)(29 -2n) + 2 \mid$.
\newline
And \ \ \ \ \ \ \ \ \ \ \ \  $ \frac {1-p}{2} < n < p-1$ implies that $ -14 < n < 28$.
\newline
For $n$ = -13 $ p_n =\  \mid (1+2n)(29 -2n) + 2 \mid$ implies that   $ p_n = \mid -1373 \mid = 1373.$ \newline
For $n$ = -12 $ p_n =\  \mid (1+2n)(29 -2n) + 2 \mid$ implies that  $ p_n = \mid -1217 \mid = 1217$. \newline
For $n$ = -11  $ p_n =\  \mid (1+2n)(29 -2n) + 2 \mid$ implies that  $ p_n = \mid -1069 \mid = 1069$.\newline 
For $n$ = -10 $ p_n =\  \mid (1+2n)(29 -2n) + 2 \mid$ implies that  $ p_n = \mid -929\mid = 929.$ \newline
For $n$ = -9 $ p_n =\  \mid (1+2n)(29 -2n) + 2 \mid$ implies that  $ p_n = \mid -797 \mid = 797$. \newline
For $n$ = -8 $ p_n =\  \mid (1+2n)(29 -2n) + 2 \mid$ implies that  $ p_n = \mid -673 \mid = 673$. \newline
For $n$ = -7  $ p_n =\  \mid (1+2n)(29 -2n) + 2 \mid$ implies that  $ p_n = \mid -557 \mid = 557$.\newline 
For $n$ = -6 $ p_n =\  \mid (1+2n)(29 -2n) + 2 \mid$ implies that  $ p_n = \mid -449\mid = 449.$ \newline
For $n$ = -5 $ p_n =\  \mid (1+2n)(29 -2n) + 2 \mid$ implies that  $ p_n = \mid -349 \mid = 349.$ \newline
For $n$ = -4 $ p_n =\  \mid (1+2n)(29 -2n) + 2 \mid$ implies that  $ p_n = \mid -257 \mid = 257.$ \newline
For $n$ = -3 $ p_n =\  \mid (1+2n)(29 -2n) + 2 \mid$ implies that  $ p_n = \mid -173 \mid = 173$. \newline
For $n$ = -2 $ p_n =\  \mid (1+2n)(29 -2n) + 2 \mid$ implies that  $ p_n = \mid -97 \mid = 97.$ \newline
For $n$ = -1  $ p_n =\  \mid (1+2n)(29 -2n) + 2 \mid$ implies that  $ p_n = \mid -29\mid = 29.$\newline 
For $n$ = 0 $ p_n =\  \mid (1+2n)(29 -2n) + 2 \mid$ implies that  $ p_n = \mid 31\mid = 31.$ \newline
For $n$ = 1 $ p_n =\  \mid (1+2n)(29 -2n) + 2 \mid$ implies that  $ p_n = \mid 127 \mid = 127.$ \newline
For $n$ = 2 $ p_n =\  \mid (1+2n)(29 -2n) + 2 \mid$ implies that  $ p_n = \mid 163\mid = 163.$ \newline
For $n$ = 3 $ p_n =\  \mid (1+2n)(29 -2n) + 2 \mid$ implies that  $ p_n = \mid 191 \mid = 191.$ \newline
For $n$ = 4  $ p_n =\  \mid (1+2n)(29 -2n) + 2 \mid$ implies that  $ p_n = \mid 211 \mid = 211.$\newline 
For $n$ = 5 $ p_n =\  \mid (1+2n)(29 -2n) + 2 \mid$ implies that  $ p_n = \mid 223\mid = 223.$ \newline
For $n$ = 6 $ p_n =\  \mid (1+2n)(29 -2n) + 2 \mid$ implies that  $ p_n = \mid 227 \mid = 227.$ \newline
For $n$ = 7 $ p_n =\  \mid (1+2n)(29 -2n) + 2 \mid$ implies that  $ p_n = \mid 223 \mid = 223$. \newline
For $n$ = 8 $ p_n =\  \mid (1+2n)(29 -2n) + 2 \mid$ implies that  $ p_n = \mid 211 \mid = 211$. \newline
For $n$ = 9 $ p_n =\  \mid (1+2n)(29 -2n) + 2 \mid$ implies that  $ p_n = \mid 191 \mid = 191.$ \newline
For $n$ = 10  $ p_n =\  \mid (1+2n)(29 -2n) + 2 \mid$ implies that  $ p_n = \mid 79\mid = 79.$\newline 
For $n$ = 11 $ p_n =\  \mid (1+2n)(29 -2n) + 2 \mid$ implies that  $ p_n = \mid 163\mid = 163.$ \newline
For $n$ = 12 $ p_n =\  \mid (1+2n)(29 -2n) + 2 \mid$ implies that  $ p_n = \mid 127 \mid = 127.$ \newline
For $n$ = 13 $ p_n =\  \mid (1+2n)(29 -2n) + 2 \mid$ implies that  $ p_n = \mid 83 \mid = 83.$ \newline
For $n$ = 14 $ p_n =\  \mid (1+2n)(29 -2n) + 2 \mid$ implies that  $ p_n = \mid 31 \mid = 31.$ \newline
For $n$ = 15 $ p_n =\  \mid (1+2n)(29 -2n) + 2 \mid$ implies that   $ p_n = \mid -29 \mid = 29.$ \newline
For $n$ = 16 $ p_n =\  \mid (1+2n)(29 -2n) + 2 \mid$ implies that  $ p_n = \mid -97 \mid = 97.$ \newline
For $n$ = 17  $ p_n =\  \mid (1+2n)(29 -2n) + 2 \mid$ implies that  $ p_n = \mid -173 \mid = 173.$\newline 
For $n$ = 18 $ p_n =\  \mid (1+2n)(29 -2n) + 2 \mid$ implies that  $ p_n = \mid -257\mid = 257.$ \newline
For $n$ = 19 $ p_n =\  \mid (1+2n)(29 -2n) + 2 \mid$ implies that  $ p_n = \mid -349 \mid = 349.$ \newline
For $n$ = 20 $ p_n =\  \mid (1+2n)(29 -2n) + 2 \mid$ implies that  $ p_n = \mid -449 \mid = 449.$ \newline
For $n$ = 21  $ p_n =\  \mid (1+2n)(29 -2n) + 2 \mid$ implies that  $ p_n = \mid -557 \mid = 557.$\newline 
For $n$ = 22 $ p_n =\  \mid (1+2n)(29 -2n) + 2 \mid$ implies that  $ p_n = \mid -673\mid = 673.$ \newline
For $n$ = 23 $ p_n =\  \mid (1+2n)(29 -2n) + 2 \mid$ implies that  $ p_n = \mid -797 \mid = 797.$ \newline
For $n$ = 24 $ p_n =\  \mid (1+2n)(29 -2n) + 2 \mid$ implies that  $ p_n = \mid -929 \mid = 929.$ \newline
For $n$ = 25 $ p_n =\  \mid (1+2n)(29 -2n) + 2 \mid$ implies that  $ p_n = \mid -1069 \mid = 1069.$ \newline
For $n$ = 26 $ p_n =\  \mid (1+2n)(29 -2n) + 2 \mid$ implies that  $ p_n = \mid -1217 \mid = 1217.$ \newline
For $n$ = 27  $ p_n =\  \mid (1+2n)(29 -2n) + 2 \mid$ implies that  $ p_n = \mid -1373\mid = 1373.$\newline
\end{proof}
For each prime $p>2$ that is a lower member of a pair of twin primes less than $41$, equation (5) generates consecutively prime numbers, $\frac {3p -5}{2}$ times. In total,
the expression has produced primes 85 times without interruption. But these primes are often repeated.

\makeatletter
\renewcommand{\@biblabel}[1]{[#1]\hfill}
\makeatother

\end{document}